\documentclass[12pt]{amsart}
\usepackage{amssymb}
\usepackage{amsmath}
\numberwithin{equation}{section}

\usepackage{amsfonts, amsmath, amsthm, amssymb}

\newtheorem{thm}{Theorem}[section]
\newtheorem{cor}[thm]{Corollary}
\newtheorem{lem}[thm]{Lemma}
\newtheorem{prop}[thm]{Proposition}
\newtheorem{exm}[thm]{Example}
\newtheorem{defn}[thm]{Definition}
\newtheorem{rem}[thm]{Remark}
\numberwithin{equation}{section}

\def\cc{\mathcal{ C}}

\def\cf{\mathcal{ F}}

\def\cv{\mathcal {V}}

\def\eb{{\mathbf{e}}}

\def\bbf{{\mathbb F}}

\def\bi{{\mathbb I}}

\def\bn{{\mathbb N}}
\def\bp{{\mathbb P}}

\def\a{\alpha}
\def\b{\beta}
  \def\G{\Gamma}
\def\d{\delta}

\def\l{\lambda}

\def\xb{{\mathbf{x}}}
\def\yb{{\mathbf{y}}}

\def\yb{{\mathbf{y}}}

\def\tb{{\mathbf{t}}}

\def\a{\alpha}
\begin{document}
    \title[Infiniteness of Orthogonal Preserving QSO]{Infinite Dimensional Orthogonal Preserving Quadratic Stochastic Operators}

   \author{Farrukh Mukhamedov}
\address{Farrukh Mukhamedov\\
    ,
 Department of Mathematical Sciences, \& College of Science,\\
The United Arab Emirates University, Al Ain, Abu Dhabi,\\
15551, UAE} \email{{\tt far75m@yandex.ru}, {\tt
farrukh.m@uaeu.ac.ae}}

\author{Ahmad Fadillah Embong}
\address{Ahmad Fadillah\\
 Department of Computational \& Theoretical Sciences\\
Faculty of Science, International Islamic University Malaysia\\
Kuantan, Pahang, Malaysia} \email{{\tt ahmadfadillah.90@gmail.com}}

    \begin{abstract}
In the present paper, we study infinite dimensional orthogonal
preserving quadratic stochastic operators (OP QSO). A full
description of OP QSOs in terms of their canonical form and heredity
coefficient's values is provided. Furthermore, some properties of OP
QSOs and their fixed points are studied.

    \end{abstract}

\maketitle

\section{Introduction}

The history of quadratic stochastic operators (QSOs) is traced back
to Bernstein's work \cite{1} where such kind of operators appeared
from the problems of population genetics (see also \cite{11}). These
kind of operators describe time evolution of variety species in
biology and are represented by so-called Lotka-Volterra(LV) systems
\cite{29}, but currently in the present, there are many papers
devoted to these operators owing to the fact that they have
plentiful applications especially in modelings in many different
fields such as biology \cite{8,20} (population and disease
dynamics), physics \cite{21,25}(non-equilibrium statistical
mechanics) , economics, and mathematics \cite{11,20,25} (replicator
dynamics and games).

A quadratic stochastic operator is usually used to present the time
evolution of species in biology, which arises as follows. By
considering an evolution of species in biology as given in the
situation where  $I = \{1,2,\dots,n\}$ is the $n$ type of species
(or traits) in a population, the probability distribution of the
species in an early state of that population is $x^{(0)} =
(x_1^{(0)},\dots,x_n^{(0)})$. On a side note, we define $P_{ij,k}$
as the probability of an individual in the $i^{th}$ species and
$j^{th}$ species to cross-fertilize and produce an individual from
$k^{th}$ species (trait). Given $x^{(0)} =
(x_1^{(0)},\dots,x_n^{(0)})$, we can find the probability
distribution of the first generation, $x^{(1)} =
(x_1^{(1)},\dots,x_n^{(1)})$ by using a total probability, i.e.,
\begin{eqnarray*}
    x_k^{(1)}=\sum\limits_{i,j=1}^n P_{ij,k}x_i^{(0)}x_j^{(0)}, \ \
    k\in\{1,\dots,n\}. \label{operator}
\end{eqnarray*}

This relation defines an operator which is denoted by $V$ and it is
called \textit{quadratic stochastic operator (QSO)}. Each QSO maps
the simplex $S^{n-1} = \{\textbf{x}\in \mathbb{R}^{n}\quad |x_i\geq
0,\quad \sum\limits_{i=1}^{n}x_i = 1\}$ into itself. Moreover, the
operator $V$ can be interpreted as an evolutionary operator that
describes the sequence of generations in terms of probability
distributions if the values of $P_{ij,k}$ and the distribution of
the current generation are given. The most well-known class in the
theory QSO is a Volterra one, namely whose heredity coefficients
satisfy
\begin{eqnarray}
    \label{eqn_cond_Vol_lr}
    P_{ij,k} =0 \textmd{ if } k \notin \{i,j\}.
\end{eqnarray}
The condition \eqref{eqn_cond_Vol_lr}, biologically, means that each
individual can inherit only the species of the parents. The dynamics
of Volterra QSO was studied in \cite{R.gani_tournment,GGZ}.
Nevertheless, not all QSOs are of Volterra-type, therefore, the
understanding of the dynamics of non-Volterra QSO still remains
open. We refer the reader to \cite{6,MG2015} as the exposition of
the recent achievements and open problems in the theory of the QSO
can be further researched.

One of the main problems in the theory of nonlinear operator is to
study the limiting behavior of nonlinear operators. To this day,
there are a handful of studies dedicated to the exploration of the
dynamics of higher dimensional systems despite the fact that it is a
very exquisite and important topic. Although, most research has been
focused on the simplex $S^{n-1}$, but there are models where the
probability distribution is given on a countable set, which means
that the corresponding QSO is defined on an infinite-dimensional
space.

The simplest case of the infinite-dimensional space is the Banach
space $ \ell_{1} $ of absolutely summable sequences. It is worth
mentioning that some infinite dimensional QSOs were studied in
\cite{J2013,M2000,Far_Has_Temir}.

On the other hand, from \cite{man(2016)2d} with the results of
\cite{taha} we conclude that a QSO (acting on finite dimensional
simplex) is surjective, if and only if, it is orthogonal preserving
(OP) QSO. Here by the orthogonality of distributions we mean their
disjointness. We cannot afford to ignore the surjectivity of a
quadratic operator is strongly tied up with nonlinear optimization
problems \cite{Arut2012}. Furthermore, any orthogonal preserving QSO
is a permutation of Volterra QSO in\cite{Has_Far_OPQSO_e,taha}. Yet,
if we look at the same problem in the infinite dimensional setting,
the last statement becomes incorrect. Also in
\cite{Has_Far_OPQSO_e}, we have considered a special class of
orthogonal preserving operators for which an analogous result was
obtained replicated in the finite dimensional setting.
Unfortunately, this type of result is wrong in a general setting.
Therefore, in this paper, we go on a voyage of discovery in an
attempt to describe the orthogonality preserving infinite
dimensional quadratic stochastic operators in a general case. We
notice that every linear stochastic operators can be considered as a
particular case of QSO. In the later case, there are many papers tha
are devoted to the orthogonal preserving linear operators defined on
various Banach spaces (see for example \cite{AR,Bur,F,FS,Horn,W}),
once the nonlinearity appears in operators, then all existing
methods (for linear operators) are no longer applicable. The
simplest nonlinearity is quadratic which for these kinds of we fully
describe, OP QSOs in terms of the their heredity coefficients, and
provide their canonical forms. Last but not least, we provide
ceratin examples of such kind of operators along with the properties
of OP QSOs and their fixed points.

\section{Orthogonal Preserving QSO}

Let $E$ be a subset of $\bn$. Denote
\[ S^{E} = \left\{\textbf{x}=(x_i)_{\in E}\in \mathbb{R}^{E}\ : \ x_i\geq 0,\quad \sum\limits_{i\in E}x_i = 1\right\}. \]

In what follows, by $\eb_i$ we denote the standard basis in $S^E$,
i.e. $\eb_i=(\d_{ik})_{k\in E}$ ($i\in E$), where $\d_{ij}$ is the
Kroneker delta.

Let $V$ be a mapping defined by
\begin{eqnarray}\label{eqn_qso}
    V(\xb)_{k} = \sum\limits_{i,j\in E}P_{ij,k}x_ix_j, \ \ k
    \in E
\end{eqnarray}
here, $\{P_{ij,k}\}$ are hereditary coefficients which satisfy
\begin{eqnarray}\label{eqn_coef_cond}
    P_{ij,k}\geq0, \quad P_{ij,k}=P_{ji,k}, \quad
    \sum\limits_{k\in E1}P_{ij,k} = 1, \quad i,j,k\in E
\end{eqnarray}
One can see that $V$ maps $S^E$ into itself and $V$ is called
\textit{Quadratic Stochastic Operator (QSO)} \cite{M2000}.

By \textit{support} of $ \xb = (x_{i})_{i\in E}\in S^E $ we mean a
set $ Supp(\xb) = \left\{ i \in E \ : \  x_{i} \neq 0 \right\}. $ A
sequence  $ \{A_{k}\} $ of sets is called \textit{cover} of a set $
B $ if $ \bigcup\limits_{k=1}^{\infty}A_{k} = B$ and $ A_{i}\cap
A_{j} = \emptyset $ for $ i,j \in \bn $ ($i\neq j$).

Recall that two vectors $ \xb=(x_k), \yb=(y_k)$ belonging to $S^E $
are called \textit{orthogonal} (denoted by $ \xb \perp \yb $) if $
supp(\xb) \cap supp(\yb) = \emptyset $. If $ \xb,\yb\in S^E$, then
one can see that  $ \xb \perp \yb $ if and only if $ \xb \circ \yb =
0 $ (or $ x_{k} \cdot y_{k} =0 $ for all $ k \in E$). Here, $ \circ
$ stands for the standard dot product.

\begin{defn}
    A QSO $V$ given by \eqref{eqn_qso} is called \textit{orthogonal
        preserving QSO (OP QSO)} if for any $ \xb,\yb \in S $ with $ \xb
    \perp \yb $ one has $ V(\xb) \perp V(\yb) $.
\end{defn}

 Let $\mathbb{T}$ be a stochastic matrix given by
$(t_{ij})_{i,j\in E}$, where $t_{ij}\geq 0$, $\sum_{j\in E}t_{ij}=1$
($i\in\bn$). Then one can define a linear operator (which is called
linear stochastic operator (LSO))
\begin{eqnarray}\label{TT}
(T\xb)_k=(\xb\mathbb{T})_k=\sum_{i\in E}t_{ik}x_i, \ \ \xb\in S^E, \
\ k\in E.
\end{eqnarray}
Due to stochasticity of $\mathbb{T}$, the operator $T$ maps $S^E$
into itself. Note that each LSO can be considered as a particular
case of q.s.o. Indeed, let us define
\begin{equation}\label{VT}
 P_{ij,k}^{(T)}=\frac{t_{ik}+t_{jk}}{2}.
\end{equation} Then one can see that  $\{P_{ij,k}^{(T)}\}$ satisfies
\eqref{eqn_coef_cond}, and for the corresponding q.s.o. $V_T$ we
have
\begin{eqnarray*}
(V_T(\xb))_k&=&\sum_{i,j\in E} P_{ij,k}^{(T)}x_ix_j\\[2mm]
&=&\sum_{i,j\in E} \bigg(\frac{t_{ik}+t_{jk}}{2}\bigg)x_ix_j\\[2mm]
&=&\sum_{i\in E}t_{ik}x_i\\
&=&(T\xb)_k, \ \ \forall k\in\bn,
\end{eqnarray*}
i.e.  $V(\xb)=T\xb$ for all $\xb\in S^E$. This implies that all
results holding for QSO are valid for LSO.

\begin{rem}\label{OPT}
Let $T$ be a LSO, then for any $ \xb=(x_k) \in S^E $, $ T $ can be
written as follows
    \[ T(\xb) = \sum_{k\in E} x_k T(\eb_k). \]
Therefore, a LSO $ T $ defined on $S^E$ is orthogonal preserving if
and only if $T(\eb_{i})\perp T(\eb_{j})$ for all $i\neq j$. Indeed,
it is enough to show that the last statement implies OP of $T$. Let
us take $\xb,\yb\in S^E$ such that $\xb\perp \yb$ (i.e.
$\xb\circ\yb=0$). Then from
$$
T(\xb)=\sum_{\ell\in E} x_\ell T(\eb_\ell), \ \ \ T(\yb)=\sum_{m\in
E} y_mT(\eb_m)
$$
with $T(\eb_\ell)\circ T(\eb_m)=\d_{\ell m}$ we find
\begin{eqnarray*}
T(\xb)\circ T(\yb)&=&\bigg(\sum_{\ell\in E} x_\ell
T(\eb_\ell)\bigg)\circ \bigg(\sum_{m\in E}y_m
T(\eb_m)\bigg)\\[2mm]
&=&\sum_{\ell,m\in E} x_\ell y_m T(\eb_\ell)\circ
T(\eb_m)\\[2mm]
&=&\sum_{\ell\in E} x_\ell y_\ell\\
&=&0
\end{eqnarray*}
Now using \eqref{TT} we conclude that $T$ is OP if and only if for
the stochastic matrix $(t_{ij})$ one has $\tb_i\perp \tb_j$ for all
$i,j\in\bn$ with $i\neq j$. Here $\tb_k=(t_{ki})_{i\in E}$, $k\in E
$ for all $i\neq j$.
    \end{rem}

When we consider the QSO, then similar kind of result is not valid,
but we use some ideas from the mentioned remark.

\begin{rem} We first note that if $V$ is an OP QSO, then the system
$\{V(\eb_k)\}$ is also orthogonal. Therefore, to describe OP QSO it
is enough for us just to fix this (i.e. $\{V(\eb_k)\}$) system.
Indeed, let us denote by $ \cv $ the set of all OP QSO $V$ such that
$V(\eb_{k}) = \bbf_{k} $ for some orthogonal system $ \bbf_{k} $ in
$ S$. Now, let us assume that an OP QSO $ \widetilde{V} $ such that
    $ \widetilde{V}(\tilde\bbf_{k}) = \bbf_{k}^{'} $, where $ \{
    \tilde\bbf_{k} \} $ and $ \{ \bbf_{k}^{'} \} $ are orthogonal systems in $
    S$. On the other hand, if one considers $ \{\widetilde{V}(\eb_{k})\}
    $ then the system is also has to be orthogonal in $ S $ i.e.,
   $\widetilde{V}(\eb_{k}) = \bi_{k}$, where $ \{ \bi_{k} \} $ is an orthogonal system in $
    S $. Hence $ \widetilde{V} $ is an element of $ \cv $.
\end{rem}

Recall \cite{Far_Has_Temir} that a QSO $ V : S^E \rightarrow S^E $
is called \textit{Volterra} if one has
\begin{eqnarray}\label{eqn_cond_Vol}
    P_{ij,k} =0 \textmd{ if } k \notin \{i,j\}, \ \ i,j,k\in E.
\end{eqnarray}

\begin{rem}
In \cite{Far_Has_Temir} it was given an alternative definition
Volterra operator in terms of extremal elements of $S^E$.
\end{rem}

One can check \cite{MG2015} that a QSO $V$ is Volterra if and only
if one has
 \[ (V(\xb))_{k} = x_{k} \left( 1 + \sum\limits_{i\in E}a_{ki}x_{i} \right), \ \ k\in E, \]
    where $ a_{ki} = 2P_{ik,k}-1 $ ($ i,k \in E $. One can see that $ a_{ki}=-a_{ik} $.
This representation leads us to the following definitions.

\begin{defn}
    A QSO $V : S^{E} \rightarrow S^{E} $ is called $ \pi $-Volterra if there is a permutation $ \pi $ of $E$ such that $ V $ has the following
    form
    \[ V(\xb)_{k} = x_{\pi(k)} \left( 1 + \sum\limits_{i\in E}a_{\pi(k)i}x_{i} \right) \]
    where $ a_{\pi(k)i} = 2P_{i\pi(k),k}-1 $, $ a_{\pi(k)i}=-a_{i\pi(k)} $ for any $ i,k \in E$.
\end{defn}

In \cite{Has_Far_OPQSO_e,taha} it has been proved the following
result.

\begin{thm}[\cite{Has_Far_OPQSO_e,taha}]\label{VV} Let
$E=\{1,2,\dots,n\}$ and $ V $ be a QSO on $ S^{E} $. Then the
following statements are
    equivalent:
    \begin{itemize}
        \item[(i)] $ V $ is orthogonal preserving;
        \item[(ii)] $ V $ is $ \pi $-Volterra QSO.
    \end{itemize}
\end{thm}

In what follows, for the sake of convenience we denote $S$ instead
of $S^\bn$.

\begin{rem}
We notice that the vertices of the finite simplex $ S^{E} $
($E=\{1,2,\dots,n\}$) are described by the elements $ \eb_{k} =
(\d_{ik})_{i\in E} $. Therefore, any OP QSO on $ S^{n-1} $ is a
permutated Volterra QSO (see Theorem \ref{VV}). However, if we
consider $ S $, then one can see that there are many orthogonal
systems in $S$, which differ from the system $ \{ \eb_{k} \} $. For
example
\begin{eqnarray}\label{eqn_basic_1/2}
&&\bbf_{1}^{(1/2)} = \left( \frac{1}{2}, \frac{1}{2}, 0, \dots
\right), \bbf_{2}^{(1/2)} = \left( 0,0,\frac{1}{2}, \frac{1}{2}, 0,
\dots \right),\nonumber\\[2mm]
&& \dots, \bbf_{k}^{(1/2)} = \left(
0,0,\dots,\underbrace{\frac{1}{2} }_{2k-1}, \underbrace{\frac{1}{2}
}_{2k}, 0, \dots \right),\dots
\end{eqnarray}
Another crucial moment is that for a given orthogonal system $ \{
\bbf_{k}\} $ in $S$, the set
$$ \bigcup_{k}^{\infty}
(supp(\bbf_{k})) $$ may not equal to $\bn$. For example, we have $
\cup_{k=2}^{\infty} (supp(\eb_{k}))=\bn\setminus\{1\} $. All these
make the description of OP QSOs is more challenging than the finite
dimensional setting.
\end{rem}

In \cite{Has_Far_OPQSO_e} a special class of infinite dimensional OP
 QSOs have been studied for which an analogous of Theorem \ref{VV}
 holds.

\begin{thm}[\cite{Has_Far_OPQSO_e}] \label{thm_far_has_infi}
Let $V $ be a QSO on $ S $ such that $ V(\eb_{i})=\eb_{\pi(i)} $ for
some permutation $ \pi : \bn \rightarrow \bn $. Then $ V $ is an OP
QSO if and only if $ V $ is $ \pi $-Volterra QSO.
\end{thm}

Recall that an orthogonal basis $ \{\bbf_{k}\}_{k=1}^{\infty} $ in $
S $ is called \textit{total} if for any $ \xb \in S $ one finds $ \{
\l_{i} \}_{i=1}^{\infty}, \l \geq 0 , \sum_{i=1}^{\infty} \l_{i} =1$
such that
\[ \xb = \sum\limits_{i=1}^{\infty} \l_{i}\bbf_{i} \]

\begin{thm}
    Let $\{\bbf_{k}\}_{k=1}^{\infty} $ be an orthogonal basis in $ S $. The following conditions are equivalent
    \begin{itemize}
        \item[(i)] $ \{\bbf_{i}\}_{i=1}^{\infty} $ is total;
        \item[(ii)] For every $ k \in\bn  $ one has $ |supp(\bbf_{k})| =1 $ and $ \cup_{k=1}^{\infty} supp(\bbf_{k}) =\bn $.
    \end{itemize}
\end{thm}
\begin{proof}
    (i) $ \Rightarrow $ (ii).
    Assume contrary i.e., there exists some $ k_{0} \in \bn $ such that $ |supp(\bbf_{k_{0}})| \geq 2 $.
    Now, take $ m \in supp(\bbf_{k_{0}}) $.
    If one considers $ \eb_{m} \in S $, then due to the totality of $ \{\bbf_{k}\} $ we have
    \[ \eb_{m} = \sum\limits_{i} \l_{i}\bbf_{i} \]
    This means that $ \l_{i} =0 $ for $ i \neq k_{0} $, so $ \eb_{m} = \bbf_{k_{0}} $, which contradicts
     to $ |supp(\bbf_{k_{0}})| \geq 2 $. Now, if $\cup_{k=1}^{\infty} supp(\bbf_{k}) \subset \bn $, then for
     $ \ell \in \bn \backslash \cup_{k=1}^{\infty} supp(\bbf_{k})  $, the vector $ \eb_{\ell} $ can not be represent as a convex combination of $ \{\bbf_{k}\} $.
     Hence, we infer the statement (ii).

    (ii)  $ \Rightarrow $  (i). If (ii) holds, then the system $ \{ \bbf_{k} \}_{k=1}^{\infty} $ is a permutation of the standard basis
    $ \{ e_{k} \}_{k=1}^{\infty} $, which is clearly total.
\end{proof}

From the last theorem and Theorem \ref{thm_far_has_infi} we conclude
the following result.
\begin{cor}
    Let $ \{ \bbf_{k} \}_{k=1}^{\infty}$ be an orthogonal system in $ S $ and $ V $ is an OP QSO on $S$ such that $ V(\eb_{k}) = \bbf_{k} $,  for all $ k \in \bn $.
    Then the following statements are equivalent
    \begin{itemize}
        \item[(i)] $ V$ is an $ \pi- $Volterra QSO;
        \item[(ii)] $ \bbf_{k} $ is total.
    \end{itemize}
\end{cor}

\section{Description of OP QSOs}

In this section, we are going to describe infinite dimensional OP
QSOs.

Let $V$ be a QSO on $S$ whose heredity coefficients are
$\{P_{ij,k}\}$. Let us introduce the following vectors
\[ \bp_{ij} = (P_{ij,1}, \cdots, P_{ij,n}, \cdots) \ \ \textmd{ for any } i,j \in \bn\]
One can see that for every $i,j\in\bn$ the vector $\bp_{ij}$ belongs
to $S$. Next result describes OP QSOs in terms of the vectors
$\{\bp_{ij}\}$.

\begin{thm}\label{PPV1}
    Let $ V $ be a QSO. Then the following conditions are
    equivalent:
    \begin{itemize}
        \item[(i)] $ V $ is an OP QSO;
        \item[(ii)] For any $ A,B \subset \bn $ with $ A \cap B = \emptyset $ one has $ \bp_{ij} \perp \bp_{uv} $ for all $ i,j \in A $ and $ u,v \in B
        $.
    \end{itemize}
\end{thm}

\begin{proof} (i)$\Rightarrow$(ii). Take any $A,B \subset \bn $ with $ A \cap B = \emptyset$. Then chose two elements $\xb,\yb\in S$ such that
$supp(\xb)=A$ and $supp(\yb)=B$. From the condition  $A \cap B
=\emptyset$ one concludes that $ \xb\perp \yb $.

From the definition of QSO, we have
    \begin{eqnarray}\label{eqn_form_V_in_supp}
    V(\xb) = \left( \sum\limits_{i,j \in supp(\xb)}P_{ij,k}x_{i}x_{j} \right)_{k=1}^{\infty},  &
    V(\yb) = \left( \sum\limits_{u,v \in supp(\yb)}P_{uv,k}y_{u}y_{v} \right)_{k=1}^{\infty}.
    \end{eqnarray}
Due to the orthogonal preserving property of $ V $ one has $ V(\xb)
\circ V(\yb) =0 $, therefore one gets
    \begin{eqnarray}\label{Vperp}
    V(\xb)\circ  V(\yb)&=& \sum\limits_{k=1}^{\infty}
    \left( \sum\limits_{i,j \in supp(\xb)}P_{ij,k}x_{i}x_{j} \right)
    \left( \sum\limits_{u,v \in supp(\yb)}P_{uv,k}y_{u}y_{v} \right) \nonumber \\
    &=& \sum\limits_{i,j \in supp(\xb)} \sum\limits_{u,v \in supp(\yb)}
    \left( \sum\limits_{k=1}^{\infty} P_{ij,k} P_{uv,k} \right)
    x_{i}x_{j}y_{u}y_{v}\nonumber\\[2mm]
    &=&0
    \end{eqnarray}

According to $ i,j \in supp(\xb) $ and $ u,v \in supp(\yb) $ (i.e.,
$ x_{i}>0, y_{u}>0 $ for any $ i \in supp(\xb) $ and $ u \in
supp(\yb) $) from the last equalities, we conclude that
    \[ \sum\limits_{k=1}^{\infty} P_{ij,k} P_{uv,k} = 0 \]
which means $\bp_{ij} \circ \bp_{uv} =0$ for all $ i,j \in A $ and $
u,v \in B$.

Now let us prove (ii)$\Rightarrow$ (i).
    Now, take $ \xb, \yb \in S $ such that $ \xb \perp \yb $, then from \eqref{Vperp} one finds
     \begin{eqnarray}
    V(\xb) \circ V(\xb) = \sum\limits_{i,j \in supp(\xb)} \sum\limits_{u,v \in supp(\yb)}
    \left(\bp_{ij}\circ\bp_{uv}\right) x_{i}x_{j}y_{u}y_{v}
    \end{eqnarray}
    Due the fact $ supp(\xb) \cap supp(\yb) = \emptyset $
    and the assumption (ii) we immediately obtain $ V(\xb) \circ V(\xb) =0$, i.e. $V(\xb) \perp
    V(\xb)$. This completes the proof.
\end{proof}

From this theorem we immediately get the following corollary.

\begin{cor}\label{PPV2}
    Let $ V $ be an OP QSO, then for any $i\neq j$ ($i,j\in\bn$) one
    has $ \bp_{ii} \perp \bp_{jj} $.
\end{cor}

\begin{rem} If a QSO $V$ is given by a stochastic matrix (see
\eqref{VT}) then from Corollary \ref{PPV2} we infer that $V$ is OP
if and only if $\tb_i\perp \tb_j$ for all $i,j\in\bn$ ($i\neq j$).
This recovers the result of Remark \ref{OPT}.
\end{rem}

One can infer that from Theorem \ref{PPV1} it is difficult to write
representation of OP QSO. Therefore, for a given OP QSO $V$ we
denote $\bbf_{k}=V(\eb_k)$, $k\in\bn$. The system $\cf=\{\bbf_{k}\}$
is orthogonal. In what follows, we denote $ \bbf_{k} =
(f_{k,i})_{i\in\bn}$. One can see that $f_{k,i}=0$ if $i\notin
supp(\bbf_{k})$.



 Henceforth, $
 |A| $ is referred to the cardinality of a set $ A $ and denote
 \[ \cc_{\cf}=\bn\setminus\bigg(\bigcup_{k\in\bn} supp(\bbf_{k})\bigg)
 \]

\begin{thm}\label{thm_des_OP_su_neq_N} Let $\cf=\{\bbf_{k}\}$ be an
orthogonal system and $ V $ be a QSO on $S$ such that $ V(\eb_{k}) =
\bbf_{k}$, $k \in \bn$. Then, $ V(\xb) $ is an OP QSO if and only if
it has the following form: for any $ \xb \in S $
\begin{itemize}
\item[(a)] for any $ m \in supp(\bbf_{k}) $ \[ (V(\xb))_{m} = x_{k} \left( f_{k,m} + \sum\limits_{i=1}^{\infty} a_{ik}^{(m)}x_{i} \right ) \]
where $ a_{ik}^{(m)} = 2 P_{ik,m} - f_{k,m} $ and set  $
a_{kk}^{(m)}=0 $.
\item[(b)] for any $ c \in \cc_\cf $, $ V(\xb)_{c} $ takes one of the following form
        \begin{itemize}
            \item[(I)] if there is no $ P_{ij,c} > 0 $ for every $ i,j \in \bn $, then $ V(\xb)_{c}=0 $ or
            \item[(II)] if there exists at least one $ P_{i_{c}j_{c},c} > 0 $, then $ V(\xb)_{c} $ has one of the following form:
            \begin{itemize}
                \item[(i)] if there is no $ P_{ij,c}>0 $ for $j \in \{i_{c},j_{c} \} $ where $ i \in \bn \backslash \{j\} $, then
                \[ V(\xb)_{c} = 2 P_{i_{c}j_{c},c}x_{i_{c}}x_{j_{c}} \]
                \item[(ii)] if there exists $ P_{i_{c_{0}}j,c}>0 $ for either $ j=i_{c} $ or $ j=j_{c} $ (here let $ j=i_{c} $), then $ V(\xb)_{c} $ has one of the following form:
                \begin{itemize}
                    \item[(1)] if $ P_{i_{c_{0}}j_{c},c}>0 $ then
                    \[ V(\xb)_{c}=2\left( P_{i_{c}j_{c},c}x_{i_{c}}x_{j_{c}} + P_{i_{c_{0}}j_{c},c}x_{i_{c_{0}}}x_{j_{c}} + P_{i_{c_{0}}i_{c},c}x_{i_{c_{0}}}x_{i_{c}} \right)  \]
                    \item[(2)] if $ P_{i_{c_{0}}j_{c},c}=0 $ then
                    \[ V(\xb)_{c}= 2 x_{i_{c}} \left( P_{i_{c}j_{c},c}x_{j_{c}}+ \sum\limits_{\stackrel{i=1}{i \neq i_{c},j_{c}}}^{\infty}P_{ii_{c},c}x_{i} \right)  \]
                \end{itemize}
            \end{itemize}
        \end{itemize}
\end{itemize}
\end{thm}
\begin{proof} Let us start with "if" part, i.e. we assume that $ V $ is an OP QSOs. From the assumption $ V(\eb_{k})
= \bbf_{k} $ and the definition of QSO we have
\[ V(\eb_{k})  = (P_{kk,1}, \dots, P_{kk,m}, \dots )  = \bbf_{k} \]
This implies that
\begin{eqnarray}\label{eqn_P_kk_neq_N}
P_{kk,m} =
\left\{
\begin{array}{ll}
0 & \textmd{ if } m \notin supp(\bbf_{k}), \\
f_{k,m} & \textmd{ if } m \in supp(\bbf_{k}),
\end{array}
\right.
\end{eqnarray}
By choosing
\begin{eqnarray}\label{eqn_cho_x_k}
\xb_{k} = (x_1, \dots, x_{k-1},0, x_{k+1}, \dots) \textmd{ such that } x_{i}>0, \textmd{ for  } i \in \bn \backslash \{k\}
\end{eqnarray}
and $ \eb_{k} $ one has $ \xb_{k} \perp \eb_{k} $. Due to the
assumption, we infer that $ V(\xb_{k}) \perp V(\eb_{k}) $.
It is clear that
\[ V(\xb_k) = \left( \sum\limits_{i,j \neq k}^{\infty} P_{ij,1}x_{i}x_{j}, \dots, \sum\limits_{i,j \neq k}^{\infty} P_{ij,m}x_{i}x_{j}, \dots  \right). \]
Thus, from the fact $ V(\xb_{k}) \circ V(\eb_{k}) = 0 $ and
\eqref{eqn_cho_x_k}, we immediately find
\[ \sum\limits_{ij \neq k}^{\infty
} P_{ij,m} x_{i}x_{j} = 0 \ \ \Rightarrow P_{ij,m} =0 \textmd{ for any } i,j \neq k \ \textmd{ and } \ m \in supp(\bbf_{k}). \]
Hence, for any $  m\in supp(\bbf_{\pi(k)}) $ and for any $ \xb \in S $
\begin{eqnarray}
V(\xb)_{m} &=& \sum\limits_{i,j=1}^{\infty}P_{ij,m}x_{i}x_{j} \nonumber \\
&=& P_{kk,m}x_{k}^{2}+ \sum\limits_{i \neq k }^{\infty} P_{ik,m}x_{i}x_{k} + \sum\limits_{j \neq k }^{\infty} P_{kj,m}x_{j}x_{k} \nonumber
\end{eqnarray}
Keeping in mind $ P_{ik,m} = P_{ki,m}, x_{k}=1-\sum\limits_{i\neq
k}^{\infty}x_{i}  $ and \eqref{eqn_P_kk_neq_N}, $ (V(\xb))_{m} $
reduces to
\begin{eqnarray}
V(\xb)_{m} = x_{k}\left( f_{k,m} + \sum\limits_{i \neq k }^{\infty} \left(2P_{ik,m} - f_{k,m}\right)x_{i}\right) \nonumber
=  x_{k} \left( f_{k,m} + \sum\limits_{i=1}^{\infty} a_{ik}^{(m)}x_{i} \right )
\end{eqnarray}
which shows (a).

Next, let us consider $ c \in \cc_\cf$. Then
\begin{eqnarray}
(V(\xb))_{c} &=& \sum\limits_{i,j=1}^{\infty}P_{ij,c}x_{i}x_{j} \nonumber \\
&=& \sum\limits_{i=1}^{\infty}P_{ii,c}x_{i}^{2} + \sum\limits_{i\neq 1}^{\infty}P_{i1,c}x_{1}x_{i}+ \dots + \sum\limits_{i\neq n}^{\infty}P_{in,c}x_{n}x_{i}+ \dots \nonumber \\
&=& \sum\limits_{i=1}^{\infty}P_{ii,c}x_{i}^{2} + 2\sum\limits_{i= 2}^{\infty}P_{i1,c}x_{1}x_{i}+ \dots + 2\sum\limits_{i= n}^{\infty}P_{in,c}x_{n}x_{i}+ \dots \nonumber \\
&=& \sum\limits_{i=1}^{\infty}P_{ii,c}x_{i}^{2} + 2\sum\limits_{j= 1}^{\infty}\sum\limits_{i= j+1}^{\infty}P_{ij,c}x_{i}x_{j} \nonumber
\end{eqnarray}
Taking into account
\eqref{eqn_P_kk_neq_N}, one gets $ 
P_{kk,c}=0$ for any $ k\in \bn $ and $ c \in \cc_\cf $. Therefore
\begin{eqnarray} \label{eqn_simpl_V_c}
V(\xb)_{c} = 2\sum\limits_{j= 1}^{\infty}\sum\limits_{i= j+1}^{\infty}P_{ij,c}x_{i}x_{j}
\end{eqnarray}

First, we assume that there exist $ i_{c},j_{c} \in \bn $ such that
$ P_{i_{c}j_{c},c} > 0 $ (if it is not the case, then we get (I)
i.e., $ V(\xb)_{c}=0 $). Next, let us choose two vectors from the
simplex $ S $ as follows
\begin{eqnarray*}
&& \xb^{(i_{c}, j_{c})} = \left( 0,\dots ,0
,\underbrace{\dfrac{1}{2}}_{i_{c}^{th}}, 0,\dots,0,
\underbrace{\dfrac{1}{2}}_{j_{c}^{th}},0  \right) \\[2mm]
&&\yb^{[i_{c}, j_{c}]} = \left( y_{1}, \dots, y_{i_{c}-1}, 0,
y_{i_{c}+1},\dots, y_{j_{c}-1},0,y_{j_{c}+1}, \dots \right)
\end{eqnarray*}
 where $
y_{i}>0 $ for any $ i\in \bn\backslash \{ i_{c},j_{c} \} $. Clearly
$ \xb^{(i_{c}, j_{c})} $ is orthogonal to $ \yb^{[i_{c}, j_{c}]} $,
hence by assumption on $ V $
\begin{eqnarray} \label{eqn_per}
V\left( \xb^{(i_{c}, j_{c})}\right) \perp V\left( \yb^{[i_{c}, j_{c}]}\right)
\end{eqnarray}
From the part (a), one gets
\[ V\left( \xb^{(i_{c}, j_{c})}\right)_{k} \cdot V\left( \yb^{[i_{c}, j_{c}]}\right)_{k} = 0 \ \ \forall \ k \ \in \bigcup\limits_{i \in \bn} supp(\bbf_{i}) \]

Using \eqref{eqn_simpl_V_c}, one has
\[ V\left( \xb^{(i_{c}, j_{c})}\right)_{c} = \dfrac{1}{2}P_{i_{c}j_{c},c} \ \ \textmd{ and } V(\yb^{[i_{c}, j_{c}]})_{c} = 2\sum\limits_{\stackrel{j=1}{\stackrel{j\neq i_{c}}{j\neq j_{c}}}}^{\infty}\sum\limits_{ \stackrel{i= j+1}{\stackrel{i\neq i_{c}}{i\neq j_{c}}} }^{\infty}P_{ij,c}y_{i}y_{j} \]

Due to \eqref{eqn_per} and the assumption $ P_{i_{c}j_{c},c}>0 $ one
infers that $ V\left( \xb^{(i_{c}, j_{c})}\right)_{c} \cdot
V(\yb^{[i_{c}, j_{c}]})_{c} = 0 $  whence
\begin{eqnarray} \label{eqn_P_ij,c=0}
P_{ij,c} = 0, \ \ \forall \ i,j \in \bn \backslash  \{ i_{c},j_{c}
\}
\end{eqnarray}

Moreover, we are interested to find the following coefficients
\[ P_{ii_{c},c}, \ \ P_{ij_{c},c} \ \ \textmd{ for all } i \in \bn \backslash \{ i_{c}, j_{c} \} \]
Furthermore, we assume, there exists $ i_{c_{0}} $ such that $
P_{i_{c_{0}}j,c}>0 $ for either $ j=i_{c} $ or $ j=j_{c} $ (here let
$ j=i_{c} $) (if it is not the case, then $ V(\xb)_{c} =
P_{i_{c}j_{c},c}x_{i_{c}}y_{j_{c}} $ which gives (i)). Without the
loss of generality, we may consider $ i_{c_{0}} < i_{c} $. Next, let
us choose
\begin{eqnarray*}
&& \xb^{(i_{c_{0}}, i_{c})} = \left( 0,\dots ,0
,\underbrace{\dfrac{1}{2}}_{i_{c_{0}}^{th}}, 0,\dots,0,
\underbrace{\dfrac{1}{2}}_{i_{c}^{th}},0  \right)\\[2mm]
&& \yb^{[i_{c_{0}}, i_{c}]} = \left( y_{1}, \dots, y_{i_{c_{0}-1}},
0, y_{i_{c_{0}+1}},\dots, y_{i_{c}-1},0,y_{i_{c}+1}, \dots \right)
\end{eqnarray*}

Using the facts from \eqref{eqn_simpl_V_c} and \eqref{eqn_P_ij,c=0},
one finds
\[ V\left( \xb^{(i_{c}, j_{c})}\right)_{c} = \dfrac{1}{2}P_{i_{c_{0}}i_{c},c} \ \ \textmd{ and } V(\yb^{[i_{c_{0}}, i_{c}]})_{c} = 2\sum\limits_{ \stackrel{i= 1}{\stackrel{i\neq i_{c_{0}}}{ \stackrel{i\neq i_{c}}{i \neq j_{c}} }} }^{\infty}P_{ij_{c},c}y_{i}y_{j_{c}} \]

Hence, by the same argument as before $ P_{ij_{c},c} = 0 $ for any $
i \in \bn \backslash \{ i_{c_{0}}, i_{c} \} $. Here, we consider two
subcases:

\textbf{Case 1.} Let $ P_{i_{c_{0}}j_{c},c} >0 $. By the same
argument as before and choosing
\begin{eqnarray*}
&&\xb^{(i_{c_{0}}, j_{c})} = \left( 0,\dots ,0
,\underbrace{\dfrac{1}{2}}_{i_{c_{0}}^{th}}, 0,\dots,0,
\underbrace{\dfrac{1}{2}}_{j_{c}^{th}},0  \right) \\[2mm]
&&\yb^{[i_{c_{0}}, j_{c}]} = \left( y_{1}, \dots, y_{i_{c_{0}-1}},
0, y_{i_{c_{0}+1}},\dots, y_{j_{c}-1},0,y_{j_{c}+1}, \dots \right)
\end{eqnarray*}
we obtain $ P_{ii_{c},c} = 0 $ for any $ i \in \bn \backslash \{
i_{c_{0}}, j_{c} \} $. Therefore, in this case, we can write $
V(\xb)_{c} $ in the  form as  given by (1).

\textbf{Case 2.} In this case, we suppose that $
P_{i_{c_{0}}j_{c},c} =0 $. Then, it is clear that we find (2).\\

Now let us turn to "only if" part. This part comes directly from the
fact $ \xb \perp \yb $, i.e. $ x_{k}\cdot y_{k}=0$  for all $ k\in
\bn $. The orthogonality of $ \xb $ and $ \yb $ implies that, for
any fixed $ k\in \bn $, either $ x_{k} =0 $ or $ y_{k}=0 $.
Therefore, if $ m\in supp(\bbf_{\pi(k)}), k \in \bn $, then from (a)
one finds $V(\xb)_{m}\cdot V(\yb)_{m} = 0$.

Using (b) one can check that we have
\[ V(\xb)_{c}\cdot V(\yb)_{c} = 0 \ \ \textmd{ for all } c\in \cc_\cf \]
This completes the proof.
\end{proof}

We point out that if $\cf=\{\bbf_{k}\}$ is an orthogonal system,
then for any injective mapping $\pi:\bn\to\bn$, the system
$\cf_\pi=\{\bbf_{\pi(k)}\}$ is also orthogonal. Hence, the previous
theorem will still remain valid for $\{\bbf_{\pi(k)}\}$.

\begin{cor}\label{thm_des_OP_su_neq_N}
Let $\cf=\{\bbf_{k}\}$ be an orthogonal system and $ V $ be a QSO
such that $ V(\eb_{k}) = \bbf_{\pi(k)} $, $k\in\bn$, for some
injective mapping $ \pi : \bn \rightarrow \bn $ Then, $ V(\xb) $ is
an OP QSO if and only if it has the following form, for any $ \xb
\in S $:
\begin{itemize}
\item[(a)] For any $ m \in supp(\bbf_{\pi(k)}) $ \[ V(\xb)_{m} = x_{k} \left( f_{\pi(k),m} + \sum\limits_{i=1}^{\infty} a_{ik}^{(m)}x_{i} \right ) \]
where $ a_{ik}^{(m)} = 2 P_{ik,m} - f_{\pi(k),m} $ and set  $
a_{kk}^{(m)}=0 $.
\item[(b)] For any $ c \in \cc_{\cf_\pi} $, $ V(\xb)_{c} $ takes one of the following form
        \begin{itemize}
            \item[(I)] If there is no $ P_{ij,c} > 0 $ for every $ i,j \in \bn $, then $ V(\xb)_{c}=0 $ or
            \item[(II)] If there exist at least one $ P_{i_{c}j_{c},c} > 0 $, then $ V(\xb)_{c} $ has one of the following form:
            \begin{itemize}
                \item[(i)] If there is no $ P_{ij,c}>0 $ for $j \in \{i_{c},j_{c} \} $ where $ i \in \bn \backslash \{j\} $, then
                \[ V(\xb)_{c} = 2 P_{i_{c}j_{c},c}x_{i_{c}}x_{j_{c}} \]
                \item[(ii)] If there exist $ P_{i_{c_{0}}j,c}>0 $ for either $ j=i_{c} $ or $ j=j_{c} $ (here let $ j=i_{c} $), then $ V(\xb)_{c} $ has one of the following form:
                \begin{itemize}
                    \item[(1)] If $ P_{i_{c_{0}}j_{c},c}>0 $ then
                    \[ V(\xb)_{c}=2\left( P_{i_{c}j_{c},c}x_{i_{c}}x_{j_{c}} + P_{i_{c_{0}}j_{c},c}x_{i_{c_{0}}}x_{j_{c}} + P_{i_{c_{0}}i_{c},c}x_{i_{c_{0}}}x_{i_{c}} \right)  \]
                    \item[(2)] If $ P_{i_{c_{0}}j_{c},c}=0 $ then
                    \[ V(\xb)_{c}= 2 x_{i_{c}} \left( P_{i_{c}j_{c},c}x_{j_{c}}+ \sum\limits_{\stackrel{i=1}{i \neq i_{c},j_{c}}}^{\infty}P_{ii_{c},c}x_{i} \right)  \]
                \end{itemize}
            \end{itemize}
        \end{itemize}
\end{itemize}
\end{cor}

An immediate consequence of the theorem is the following result.

\begin{cor}\label{cor_OP_hereCo}
Let $\cf=\{\bbf_{k}\}$ be an orthogonal system and $ V $ be a QSO
such that
    $ V(\eb_{k}) = \bbf_{\pi(k)} $, $k\in\bn$, for some injective
mapping $ \pi : \bn \rightarrow \bn $. Then $ V(\xb) $ is an OP QSO
if and only if the
    heredity coefficients  $ P_{ij,k} $ satisfy the following ones:
\begin{itemize}
\item[(a)] $ P_{ii,k} =f_{\pi(i),k} \ \textmd{for}\ k \in supp(\bbf_{\pi(i)}),\textmd{ and } \ \ P_{ij,k}=0 \ \textmd{ for } k \notin \{ supp(\bbf_{\pi(i)}) \cup supp(\bbf_{\pi(j)}) \}  $
\item[(b)] The coefficients  $ P_{ij,c} $, where $ c \in \cc_{\cf_\pi} $, satisfy one of the following ones:
\begin{itemize}
\item[(I)] $ P_{ij,c}=0 $ for all $ i,j \in \bn $ or
\item[(II)] If there exist $ P_{i_{c}j_{c},c} > 0 $, then $ P_{ij,c}=0 $ for any $ i,j \in \bn\backslash \{ i_{c},j_{c} \} $. Further, the other coefficients must satisfy one of the following,
\begin{itemize}
\item[(i)] $ P_{ij,c}=0 $ for $ j\in\{i_{c},j_{c}\} $ for all $ i \in \bn\backslash \{ j \}  $ or
\item[(ii)] If there exist $ P_{i_{c_{0}}j,c}>0 $ for either $ j=i_{c} \textmd{ or } j=j_{c} $ (here we let $ j=i_{c} $), then $ P_{ij_{c},c}=0 $ for any $ i \in \bn\backslash \{ i_{c},j_{c},i_{c_{0}} \} $. Moreover one of the following must be satisfied:

\begin{itemize}

  \item[(1)] $ P_{i_{c_{0}}j_{c},c}>0 $, then $ P_{ii_{c},c}=0 $ for any $ i \in \bn \backslash \{ i_{c},j_{c},i_{c_{0}} \} $ or
  \item[(2)] $ P_{i_{c_{0}}j_{c},c}=0 $

\end{itemize}
\end{itemize}
\end{itemize}
\end{itemize}
\end{cor}

\begin{cor}
Let $\cf=\{\bbf_{k}\}$ be an orthogonal system and $ T $ is a LSO on
$S$ such that $ T(\eb_{k}) = \bbf_{\pi(k)} $ for any $ k\in \bn $
and an injective mapping $ \pi : \bn \rightarrow \bn $, then $ T $
is an OP linear stochastic operator if and only if $ T $ takes the
following form:
\begin{itemize}
\item[(i)] For any $ m \in supp(\bbf_{\pi(k)}) $
\[ T(\xb)_{m} = f_{k,m}x_{k} \]
\item[(ii)] For any $ c \in \cc_{\cf_\pi} $
\[ T(\xb)_{c} = 0 \]
\end{itemize}
\end{cor}


\begin{rem} \label{cor_OP_N} Let $\cf=\{\bbf_{k}\}$ be an
orthogonal system. One of the important class of infinite
dimensional OP QSO is when the union of the supports of
$\{\bbf_{k}\} $ cover $ \bn $.
    So, let $ V $ be a QSO such that $ V(e_{k}) = \bbf_{\pi(k)} $ for some injective mapping $\pi:\bn\to\bn$, and $\mathcal{C}_{\cf_\pi}=\emptyset$.
    Then $ V $ is OP if and only if one has
    \begin{itemize}
    \item[(i)] $ V $ has the form given by
         \[ V(\xb)_{m} = x_{k} \left( f_{\pi(k),m} + \sum\limits_{i=1}^{\infty} a_{ik}^{(m)}x_{i} \right ) \]
         for any $ m \in supp(\bbf_{\pi(k)})$.
    \item[(ii)] The  heredity coefficients  $ P_{ij,k} $ satisfy
    \[ P_{ii,k} = f_{i,k}  \ \ \forall \  k \in supp(\bbf_{\pi(i)}) \textmd{ and } P_{ij,k}=0 \ \ \forall k \notin \{ supp(\bbf_{\pi(i)}) \cup supp(\bbf_{\pi(j)}) \} \]
    \end{itemize}
    \end{rem}

Now, it is natural to consider an orthogonal system $
\cf=\{\bbf_{k}\} $ of $ S $ such that the support of each (or some)
$ \bbf_{k} $ is countable. Let us provide an example of such kind of
orthogonal system. Take $A_j=\{j2^n\ :\ n\geq 0\}$, $j\in 2\bn-1$.
It is clear that $\{A_j\}$ is a cover for $\bn$.  Now, for each
$j\in2\bn-1$  we define $\bbf_{j}=(f^{(j)}_m)_{m=1}^\infty$ as
follows: for each $j\in 2\bn-1$ define
\begin{equation*}
f^{(j)}_m= \left\{
\begin{array}{ll}
\frac{j-1}{j}\big(\frac1j\big)^n, \ \ m=j2^n, \ n\geq 0,\\[3mm]
0, \qquad \ \qquad m\notin A_j.
\end{array}
\right.
\end{equation*}
One can see that the system $\{\bbf_{j}\}_{i\in 2\bn-1}$ is
orthogonal and $supp(\bbf_{j})=A_j$, $j\in 2\bn-1$.

Let us consider some examples of OP QSO defined on $ S $.

\begin{exm}
Now we are going to produce an example of quadratic shift operator.
Assume that a QSO $V$ such that $ V(\eb_{i}) = \eb_{i+1} $ for every
$ i\in \bn $. From Corollary \ref{cor_OP_hereCo} one gets $
P_{ii,i+1}=1 $ for any $ i \in \bn $. Choose
$ P_{i1,2}=0 $ for any $ i \geq 2 $. Next, we take for any $ k \geq
2 $
\[ 
P_{ik,k+1}=1 \textmd{ for } i\in \{ 1,2,\dots,k-1 \} \textmd{ and }
P_{ik,k+1}=0  \textmd{ for }  i \geq k+1  \] From the selected
heredity coefficients, we have $ P_{ij,1}=0 $ for any $ i,j \in \bn
$ and it is clear that they satisfy \eqref{eqn_coef_cond} hence $ V
$ is well-defined.
Thus, using Theorem \eqref{thm_des_OP_su_neq_N} one gets
\begin{eqnarray*}
(V(\xb))_k &=& \left\{
\begin{array}{l}
0, \qquad k=1, \\
x_{1}^{2}, \qquad k=2, \\
x_{k}\left( 1+ \sum\limits_{i=1, i\neq k}^{\infty}
(2P_{ik,k+1}-1)x_{i} \ \   \right), \ \ k \geq 3
\end{array}\right. \nonumber \\
\label{eqn_shift} &=& \left\{
\begin{array}{l}
0, \qquad k=1, \\
x_{1}^{2}, \qquad k=2, \\
x_{k}\left( \sum\limits_{i=1}^{k-1} 2 x_{i}  + x_{k} \ \ \forall
\right) \ \ k \geq 3
\end{array}\right. \nonumber
\end{eqnarray*}
Note that $V$ is a concrete example of nonlinear shift operator.
\end{exm}

\section{Properties of OP QSO}

In this section we are going to investigate some properties of
infinite dimensional OP QSO.

In what follows, we consider proper subsets of $\bn$, i.e.
$\a\subset \bn$ with $\a\neq\bn$. For a given $ \a \subset \bn $, we
denote
\[ \G_{\a} = \{ \xb \in S :\  x_{i}=0, \ \forall i \notin \a \}, \ \ \
 ri\G_{\a} = \{ \xb \in \G_{\a} :\  x_{i}>0, \ \forall i \in\a \}
 \]
 By $Fix(V) $ we denote the set of all fixed points of $ V $, i.e.
$ Fix(V)=\{\xb\in S: \ V(\xb)=\xb\}. $ Let $ \cf=\{\bbf_{k}\}$ be an
orthogonal system of $S$.
By $ \cv_{\cf}$ we denote the set of all OP QSO which are generated
by the orthogonal system $ \cf$, i.e.  $ V \in \cv_{\cf} $ means $
V(\eb_{k}) = \bbf_{k} $ for any $ k \in \bn $.

Denote
\[ supp(\cf) = \bigcup\limits_{k=1}^{\infty}supp(\bbf_{k}) \]

\begin{lem} \label{lem_map_face} Let $\cf=\{\bbf_{k}\}$ be an orthogonal
system such that $supp(\cf)=\bn $ and $ V \in \cv_{\cf} $. Then for
any $ \a \subset \bn $ one has
\begin{itemize}
\item[(i)]  $V(\G_{\a}) \subset \G_{\a^{'}}$;

\item[(ii)] $V(ri\G_{\a}) \subset ri \G_{\a^{'}}$,
\end{itemize}
where $$ \a^{'} = \bigcup\limits_{\ell \in \a } supp(\bbf_{\ell}).
$$
\end{lem}

\begin{proof} (i) Let $ V \in \cv_{\cf} $, then due to Remark \ref{cor_OP_N} $ V $ takes the following form
\begin{equation}\label{VV-1} V(\xb)_{k} =  x_{\ell} \left( f_{\ell,k} +
\sum\limits_{i=1}^{\infty} a_{i\ell}^{(k)}x_{i} \right)
\end{equation}
for any $ k \in supp(\bbf_{\ell}) $,
 $ \ell \in \bn $.

 Now let $ \xb = (x_{1},x_{2},\dots) \in \G_{\a} $, then  $ x_{\ell} = 0 $ for any $\ell \notin \a
 $, hence from \eqref{VV-1} one finds
\begin{eqnarray} \label{eqn_V=0}
V(\xb)_{k} = 0 , \ \ \textrm{for all} \ \ k \in supp(\bbf_{\ell}), \
\ \ell\notin\a,
\end{eqnarray}
this is the assertion (i).

Now take  $ \xb = (x_{1},x_{2},\dots) \in ri\G_{\a} $, then
$x_\ell>0$ for all $ \ell \in \a $. From \eqref{VV-1} one gets
\begin{eqnarray}\label{eqn_V_more_or_eq_0}
V(\xb)_{k}
&=&  x_{\ell} \left( f_{\ell,k} + \sum\limits_{i=1}^{\infty} a_{i\ell}^{(k)}x_{i} \right) \nonumber \\
&=& x_{\ell} \left( f_{\pi\ell,k} + \sum\limits_{\stackrel{i=1}{i \neq \ell}}^{\infty} \left(2 P_{i\ell,k} - f_{\ell,k} \right) x_{i} \right) \nonumber \\
&=& x_{\ell} \left( f_{\ell,k} + \sum\limits_{\stackrel{i=1}{i \neq \ell}}^{\infty}2 P_{i\ell,k} x_{i} - f_{\ell,k} \sum\limits_{\stackrel{i=1}{i \neq \ell}}^{\infty}  x_{i} \right) \nonumber \\
&=& x_{\ell} \left( f_{\ell,k} + \sum\limits_{\stackrel{i=1}{i \neq
\ell}}^{\infty}2 P_{i\ell,k} x_{i} - f_{\ell,k} (1-x_{\ell})
\right)  \nonumber \\
&=& x_{\ell} \left( \sum\limits_{ i \in \a}2 P_{i\ell,k} x_{i} +
f_{\ell,k} x_{\ell} \right)\nonumber\\[2mm]
& \geq& f_{\ell,k} x_{\ell}^{2} > 0
\end{eqnarray}
for $ k \in supp\left( \bbf_{\ell} \right) $. This means
$V(ri\G_{\a}) \subset ri \G_{\a^{'}}$. Moreover, using
\eqref{eqn_V_more_or_eq_0} we have $$ supp(V(\xb)) =
\bigcup\limits_{\ell \in \a } supp(\bbf_{\ell}). $$ This completes
the proof.
\end{proof}

Now it is natural to consider the case $ supp(\cf) \subset \bn $.
According to Theorem \ref{thm_des_OP_su_neq_N}, for any $ c\in
\cc_\cf $ (here as before, $\cc=\bn\setminus supp(\cf)$), $
V(\xb)_{c} $ takes one of the following form
\begin{eqnarray}\label{eqn_form_c}
\left\{
\begin{array}{llll}
(i) &   V(\xb)_{c}&=& 0 \\
(ii) & V(\xb)_{c} & =& 2 P_{i_{c}j_{c},c}x_{i_{c}}x_{j_{c}} \\
(iii) & V(\xb)_{c}& =& 2\left( P_{i_{c}j_{c},c}x_{i_{c}}x_{j_{c}} + P_{i_{c_{0}}j_{c},c}x_{i_{c_{0}}}x_{j_{c}} + P_{i_{c_{0}}i_{c},c}x_{i_{c_{0}}}x_{i_{c}} \right) \\
(iv) & V(\xb)_{c} & = & 2 x_{i_{c}} \left( P_{i_{c}j_{c},c}x_{j_{c}}+ \sum\limits_{\stackrel{i=1}{i \neq i_{c},j_{c}}}^{\infty}P_{ii_{c},c}x_{i} \right)
\end{array}
\right.
\end{eqnarray}
From now on, let us keep the notation that we have used in Theorem
\ref{thm_des_OP_su_neq_N} (i.e., $ i_{c},j_{c}, i_{c_{0}} $). To get
an analogous result like in Lemma \ref{lem_map_face}, it is enough
for us to study the coordinates belonging to $ \cc_\cf$ while  $
V(\xb)_{c} $ takes one of the forms given by (ii), (iii) and (iv),
since the case $ m \in supp(\bbf_{k}) $ is already described by
Lemma \ref{lem_map_face}.

Let us take $ \a \subset \bn $. Now we consider the mentioned cases
one by one.

\textbf{CASE (ii)}. In this case, we have the following
possibilities:
\[ (I) \ i_{c}, j_{c} \in \a; \ \ (II) \ i_{c} \in \a, j_{c} \notin \a; \ \ (III) \ j_{c} \in \a, i_{c} \notin \a; \ \ (IV) \ i_{c}, j_{c} \notin \a. \]

\textbf{CASE (iii)}. In this case, we have the following ones:
\[ (I) \ i_{c}, j_{c}, i_{c_{0}} \in \a; \ \
(II) \ i_{c} \in \a \ j_{c}, i_{c_{0}} \notin \a; \ \
(III) \ i_{c},j_{c} \in \a \ i_{c_{0}} \notin \a; \ \
(IV) \ i_{c}, i_{c_{0}} \in \a \  j_{c} \notin \a. \]
\[
(V) \ j_{c} \in \a \ i_{c}, i_{c_{0}} \notin \a; \ \
(VI) \ j_{c}, i_{c_{0}} \in \a \  i_{c} \notin \a; \ \
(VII) \ i_{c}, j_{c}, i_{c_{0}} \notin \a; \ \
(VIII) \ i_{c_{0}} \in \a \ j_{c},i_{c} \notin \a; \ \
\]

\textbf{CASE (iv)}. This case is the same like CASE (ii).

\begin{rem}\label{rem_map_supp_subset_bn}
    Let $ V \in \cv_{\cf} $ such that $ supp(\cf) \subset \bn $. For any $ \a \subset \bn $ we have the following statements:
    \begin{itemize}
        \item[(a)] Let $ c \in \cc_\cf $, then $ V(\xb)_{c} $ takes the form as given by (ii). If (I) is satisfied then $ V(\xb)_{c} > 0 $ and
        in the other cases $ V(\xb)_{c} = 0 $.
        \item[(b)] Let $ c \in \cc_\cf $ then  $ V(\xb)_{c} $ takes the form as given by (iii). If (I), (III), (IV) and (VI) are satisfied
        then $ V(\xb)_{c} > 0 $ and in the other cases $ V(\xb)_{c} = 0 $.
        \item[(c)] Let $ c \in \cc_\cf$ then  $ V(\xb)_{c} $ takes the form as given by (iv). If
        \begin{itemize}
            \item[-] (I) is satisfied then $ V(\xb)_{c} > 0 $
            \item[-] (II) is satisfied and there exist $ i_{0} \in \a $ such that $ P_{i_{0}i_{c},c}>0 $ (if not, then $ V(\xb)_{c}=0 $), then $ V(\xb)_{c}>0 $
        \end{itemize}
           In the other cases $ V(\xb)_{c} = 0 $.
    \end{itemize}
\end{rem}

Let $ V$ be a OP QSO generated by an orthogonal system $
\cf=\{\bbf_{k}\} $, i.e. $V(\eb_k)=\bbf_{k}$, $k\in\bn$. Now want to
distinguish a set where some of elements of the system $\cf$
coincides with certain elements of the standard basis. Namely, let
us denote

\[ \b = \{ k \in \bn :\  \bbf_{k}=\eb_{i}  \textmd{ for some } i \in \bn \} \]

\begin{thm}\label{thm_no_fix_p_on_boun}
Let $ V \in  \cv_{\cf} $. If $ \b= \emptyset $, then for any $ \a
\subset \bn $, one has
    \[ Fix(V) \notin \G_{\a}. \]
 Moreover, if the fixed point exists, then $ Fix(V) \in ri S $.
\end{thm}

\begin{proof}
Assume that for a fixed point $\xb_0\in S$ one has $\xb_0\in
\Gamma_\alpha$ for some $ \a \subset \bn $. This means
\begin{eqnarray} 
\begin{array}{lll}
V(\xb_{0})_{k} = 0 & if & k \notin \a, \\
V(\xb_{0})_{k} > 0 & if & k \in \a.
\end{array} \nonumber
\end{eqnarray}

Now we consider two separate cases: ($supp(\cf) = \bn $)
and ($supp(\cf) \subset \bn $).\\

{\sc Case 1}. Let us suppose ($supp(\cf) = \bn $). Since $ \xb_{0} $
is a fixed point, then one has
\begin{eqnarray}\label{eqn_(i)_C(1)}
supp(V(\xb_{0})) = \a
\end{eqnarray}

On the other hands, due to the assumption $ \b = \emptyset $ and
from Lemma \ref{lem_map_face}, we get
\[ |supp( \bbf_{k})| \geq 2 \textmd{ for } k\in \a \]
and
\[ supp(V(\xb_{0})) = \bigcup\limits_{\ell \in \a }supp\left( \bbf_{\ell} \right) \]
Therefore
\begin{eqnarray}\label{eqn_contra_supp}
|supp\left( V(\xb_{0}) \right)| > |\a|
\end{eqnarray}
which contradicts to \eqref{eqn_(i)_C(1)}. Therefore, the fixed
point cannot be in the face $ \G_{\a} $ for any $ \a \subset \bn $.

\textbf{Part 2 ($supp(\cf) \subset \bn $)}. Take any $ \a \subset
\bn $. Now we are going to consider the following three possible
cases: $ \cc_\cf \cap \a = \emptyset $, $ \cc_\cf \cap \a \neq
\emptyset, \a \not\subset \cc_\cf $, and $ \a \subseteq \cc_\cf $.

In the first case, we obtain the desired result by the same argument
as in \textbf{Part 1 }.

Now we consider the case:  $ \a \backslash \cc_\cf\neq\emptyset$.
Let $ \xb_{0} \in ri\G_{\a} $ which implies \eqref{eqn_(i)_C(1)}. On
the other hand, we have $ |supp(\bbf_{\pi(k)})| \geq 2 $ for all $ k
\in \a \backslash \cc_\cf$, therefore using Lemma \ref{lem_map_face}
one concludes that
\[ |supp(V(\xb_{0}))| = |\{ \a \cap \cc_\cf \} \cup  \bigcup\limits_{k \in \a \backslash \cc_\cf} supp(\bbf_{\pi(k)}) |> |a| \]
which contradicts to \eqref{eqn_(i)_C(1)}.

Let us turn to the last case, i.e. $ \a \subseteq \cc_\cf $. Due to
$\xb_{0}\in ri\G_{\a} $ we get  \eqref{eqn_(i)_C(1)} and
\begin{eqnarray}\label{eqn_no_fix_sum=1}
\sum\limits_{k\in \a} V(\xb_{0})_{k} = 1
\end{eqnarray}

On the other hands, by taking into account that $ \xb_{0} \in
ri\G_{\a} $ and $ P_{ii,c}=0 $ for any $ i \in \bn $, $ c \in
\cc_\cf$, then one finds
\begin{eqnarray}\label{eqn_factor}
\sum\limits_{k\in \a} V(\xb_{0})_{k}  =   \sum\limits_{k\in \a}
\sum\limits_{\stackrel{i,j \in \a}{i\neq j}} P_{ij,k}x_{i}x_{j}
  = \sum\limits_{\stackrel{i,j \in \a}{i\neq j}} x_{i}x_{j} \left( \sum\limits_{k \in \a} P_{i_{c}j_{c},k} \right)
\end{eqnarray}

Since  $ \sum_{k\in \bn} P_{ij,k} = 1 $, we then obtain
\begin{eqnarray}\label{eqn_less}
\sum\limits_{k\in \a} V(\xb_{0})_{k} &\leq&
\sum\limits_{\stackrel{i,j \in \a}{i\neq j}} x_{i}x_{j} =
\sum\limits_{i \in \a}x_{i}\left( \sum\limits_{\stackrel{j \in \a}{j
\neq i}}x_{j} \right)
\end{eqnarray}
Again $ \xb_{0} \in ri\G_{\a} $ implies
\[ \sum\limits_{\stackrel{j \in \a}{j \neq i}}x_{j} < \sum\limits_{j \in \a}x_{j} =1 \ \ \textmd{ for any } i \in \a \]
Therefore,
\begin{eqnarray}\label{eqn_less-2}
\sum\limits_{k\in \a} V(\xb_{0})_{k} \leq \sum\limits_{\stackrel{i,j
\in \a}{i\neq j}} x_{i}x_{j} < \sum\limits_{i \in \a } x_{i} = 1
\end{eqnarray}
which contradicts to \eqref{eqn_no_fix_sum=1}.

Furthermore, according to the arbitrariness of  $ \a \subset \bn $,
we infer that if a fixed point $ \xb_{0} $ exists, then $ \xb_{0}
\in ri S $. This completes the proof.
\end{proof}

\begin{rem}\label{rem_OS=Sbasis}
    Let $\cf=\{ \bbf_{k}\}$ be an orthogonal system in $S$ and  $ \a \subset \bn $. Then we have
    $$ supp(\{ \bbf_{k} \}_{k \in \a}) =supp(\{ \eb_{k} \}_{k \in \a}) $$ if and only if there is a
    permutation $\pi_\a$ of $\a$ such that  $ \{\bbf_{k} \}_{k \in \a}=\{ \eb_{\pi_\a(k)} \}_{k \in \a} $.
\end{rem}

\begin{thm}\label{thm_b_neq_empty}
    Let $ V $ be an OP QSO generated by $ V(e_{k}) = \bbf_{k} $ for any $ k \in \bn $ and let set $ \b \neq \emptyset $. Assume
    that for any $\a\subset \b$ one has
    \[ \{ \bbf_{k}\}_{k \in\a}\neq\{\eb_{\pi_\a(k)}\}_{k \in \a} \]
    for any permutation $\pi_\a$ of $\a$.
    Then for any $ \a \subset\bn $ one has
    \[ Fix(V) \notin \G_{\a} \]
    Moreover, if a fixed point exists, then $ Fix(V) \in riS $.
\end{thm}

\begin{proof} Assume that for a fixed point $\xb_0\in S$ one has $\xb_0\in
\Gamma_\alpha$ for some $ \a \subset \bn $. Without loss of
generality we may assume that $\xb_0\in ri\G_\a$. Now we consider
two possibilities $ supp( \cf )  =\bn $ and $ supp( \cf )\subset\bn
$.

\textbf{Part 1 ($ supp( \cf )  =\bn $)}. There are several
possibilities:
    \begin{itemize}
        \item[(a)] $ \a \cap \b = \emptyset $
        \item[(b)] $ \a \cap \b \neq  \emptyset, \ \a \not\subset \b $
        \item[(c)] $ \a \subset \b $
    \end{itemize}
    Cases (a) and (b) follow from the same argument as in the proof of Theorem \ref{thm_no_fix_p_on_boun}, since there exists some $ k_{0} \in \a \backslash \b $
    such that $ supp(\bbf_{k_{0}}) \geq 2 $.

   Let us consider the case (c), i.e. $ \a \subset \b $. Due to our assumption, we have
     \begin{eqnarray} \label{eqn_b_non-em-1}
    supp(\xb_{0}) = supp(V(\xb_{0})) = supp(\{\eb_{k}\}_{k \in \a}) =\a
    \end{eqnarray}
    From Lemma \ref{lem_map_face} one gets that
    \begin{eqnarray}\label{eqn_b_non-em-2}
    supp(V(\xb_{0})) = supp(\{\bbf\}_{k \in \a})
    \end{eqnarray}
    From \eqref{eqn_b_non-em-1}, \eqref{eqn_b_non-em-2} and Remark \ref{rem_OS=Sbasis} we conclude
    that there is a permutation $\pi_\a$ of $\a$ such that
    $$\{ \bbf_{k}\}_{k \in \a}=\{\eb_{\pi_\a(k)}\}_{k \in \a}.$$
    which contradicts to the assumption of the theorem.

    \textbf{Part 2 ($ supp( \cf )  \subset \bn $)}. Since we have already considered all possible situations of $ \a $ and $ \b $, therefore,
    then it is enough for us to consider the following cases:  $ \a \cap \cc_\cf =\emptyset
    $, $ \a \cap \cc_\cf \neq \emptyset, \a \not\subset \cc_\cf $ and $ \a \subset
    \cc_\cf
    $. These cases can be proceeded by the same argument as in the proof of Theorem \ref{thm_no_fix_p_on_boun}.
    This completes the proof.
    \end{proof}

Now we want provide certain examples which satisfy the conditions of
the last theorem.

    \begin{exm}
Let us consider the following orthogonal system:
  \begin{eqnarray*}
&&\bbf_{1} = \left( \dfrac{1}{2}, \dfrac{1}{2}, 0,\dots,\right), \
\bbf_{2} = \eb_{3}, \ \bbf_{3}=\eb_{4}, \bbf_{4}=\eb_{5}, \\[2mm]
&&\bbf_{n} = \left(0,\dots,0, \underbrace{\dfrac{1}{2}}_{2n-1} ,
\underbrace{\dfrac{1}{2}}_{2n}, 0,\dots,\right) \textmd{ for } n\geq
5\end{eqnarray*} Let $ V $ be generated as follows
$V(\eb_k)=\bbf_{k}$, $k\in\bn$.
        One can see that the set $ \b=\{ 2,3,4 \} $ and for any subset $ A \subset \b $ we have
        \[ \{\bbf_{k}\}_{k \in A} \neq \{ \eb_{k} \}_{k \in A} \]
        Then, due to Theorem \ref{thm_b_neq_empty} for any $ \a \subset \bn $, we have $ Fix(V) \notin \G_{\a} $.
    \end{exm}

\begin{exm} Let us consider the following orthogonal system:
  \begin{eqnarray*}
&&\bbf_{1} = \left( \dfrac{1}{2}, \dfrac{1}{2}, 0,\dots,\right), \
\bbf_{2} = \eb_{6}, \ \bbf_{3}=\eb_{7}, \bbf_{4}=\eb_{8}, \\[2mm]
&&\bbf_{n} = \left(0,\dots,0, \underbrace{\dfrac{1}{2}}_{2n-1} ,
\underbrace{\dfrac{1}{2}}_{2n}, 0,\dots,\right) \textmd{ for } n\geq
5\end{eqnarray*} Let $ V $ be generated as follows
$V(\eb_k)=\bbf_{k}$, $k\in\bn$.
    One can see that $ \b=\{ 2,3,4 \} $ and $ \cc_\cf = \{ 3,4,5 \} $. Moreover, one has for any subset $ A \subset \b $
    \[ \{\bbf_{k}\}_{k \in A} \neq \{\eb_{k} \}_{k \in A} \]
    Then, due to Theorem \ref{thm_b_neq_empty} for any $ \a \subset \bn $, we have $ Fix(V) \notin \G_{\a} $.
\end{exm}

It is well-known that an infinite-dimensional simplex $S$ is not
compact either in $ \ell_{1} $ topology, nor in a weak topology,
therefore, the existence of a fixed point of any QSO $ V $ defined
on $ S $ is not always true.

\begin{exm} Let us consider an OP QSO $V$ defined by
$$
V(x_1,x_2,\cdots,x_n,\cdots)=(0,x_1,x_2,\cdots,x_n,\cdots)
$$
where $(x_n)\in S$. It is easy to see that this operator has no
fixed points belonging to $S$.
\end{exm}

Next result provides a sufficient condition for the existence of a
fixed point of OP QSO.

\begin{prop}
Let $ V \in \cv_{\cf} $ with $ supp(\cf)=\bn$. If $ \b \neq
\emptyset $ and there exists a subset $ \a \subseteq \b $ with  $
|\a| < \infty $ such that \[ \{\eb_{\pi(k)}\}_{k \in \a} =
\{\bbf_{k}\}_{k \in \a} \] for some permutation $\pi$ of $\a$. Then
there exists a fixed point $ \xb_{0} \in \G_{\a}$.
\end{prop}
\begin{proof}
Let $ \a = \{ i_{1},\dots,i_{n}\} \subseteq \b $. By the definition
of $ V $ we infer that
\begin{eqnarray}\label{eqn_i_n_j_fix_in_f}
V(\eb_{i_k}) = \eb_{\pi(i_k)}  \textmd{ for all } \
k\in\{1,\dots,n\}
\end{eqnarray}
and
\[ V(\eb_{m}) = \bbf_{m} \textmd{ for all } m \in \bn \backslash \a\]

Due to Corollary \ref{cor_OP_N}  the operator $ V $ can be written
in the following form, for any $ \xb \in  \G_{\a} $
\begin{eqnarray}\label{eqn_opr_rep}
\left\{
\begin{array}{lll}
V(\xb)_{\pi(i)} &=& x_{i}\left( 1+ \sum\limits_{\stackrel{\ell\in \a}{\ell \neq i}} \left(   2P_{\ell i,j} -1 \right)x_{\ell} \right), \ \ i\in\a   \\
V(\xb)_{k} & = & 0 \ \ \textmd{ if }\ \ k \notin \a
\end{array}
\right.
\end{eqnarray}
 This implies that $V(\G_\a)\subset\G_\a$. The compactness of $\G_\a$ with the Brouwer fixed-point Theorem yields the existence of a fixed point
$ \xb_{0}\in\G_\a$ of $V$. This finishes the proof.
\end{proof}

Immediately from the last proposition, one concludes the following
corollary.

\begin{cor}
Let $ V \in \cv_{\cf} $ with $ supp(\cf) \subset \bn $ and $ \b \neq
\emptyset $. If $ V(\xb)_{c} = 0 $ for any $ c \in \cc_\cf $ and
there exists a subset $ \a \subseteq \b $ with  $ |\a| < \infty $
such that \[ \{\eb_{\pi(k)}\}_{k \in \a} = \{\bbf_{k}\}_{k \in \a}
\] for some permutation $\pi$ of $\a$. Then there exists a fixed
point $ \xb_{0} \in \G_{\a}$.
\end{cor}

We provide an example of OP QSO that has fixed point.
\begin{exm}
   Let us consider the following orthogonal system:
    \[ \bbf_{1}=\eb_{1}, \ \bbf_{2} = \eb_{2},\ \bbf_{n} = \left(0,\dots,0, \underbrace{\dfrac{1}{2}}_{2n-3} , \underbrace{\dfrac{1}{2}}_{2n-2}, 0,\dots,\right) \textmd{ for } n\geq 3 \]
Now let $V$ be an OP QSO such that
    \[ V(\eb_{1}) =\bbf_{2}, \ \ V(\eb_{2})=\bbf_{1}, \ \ V(\eb_{k}) = \bbf_{k-1} \ \textmd{ for } \ k \geq 3 \]
One can see that $\b=\{1,2\}$, and for a permutation $\pi(1)=2$,
$\pi(1)=2$, we have $\{\eb_{\pi(k)}\}_{k \in \b} = \{\bbf_{k}\}_{k
\in \b}$. For any $ \xb \in ri\G_{\b} $, using Corollary
\ref{cor_OP_N}, one gets
    \begin{eqnarray}\label{eqn_sys_1221}
    \left\{
    \begin{array}{lll}
    V(\xb)_{1} &=& x_{2} \left( 1+ (2P_{12,1}-1)x_{1} \right)  \\
    V(\xb)_{2} &=& x_{1} \left( 1+ (2P_{21,2}-1)x_{2} \right)
    \end{array}
    \right.
    \end{eqnarray}
    In particular, assume that $ P_{12,1} = P_{21,2} = \dfrac{1}{2} $, then clearly we have $ \left( \dfrac{1}{2} ,\dfrac{1}{2}\right) $ as a
    fixed point for the system \eqref{eqn_sys_1221}.
    Clearly, $ \left( \dfrac{1}{2}, \dfrac{1}{2}, 0,\dots, \right) \in \G_{\{ 1,2 \}} $ is a fixed point for $ V $.
\end{exm}

\section*{Acknowledgments}
The present work is supported by the UAEU "Start-Up" Grant, No.
31S259.


\begin{thebibliography}{9}

    \bibitem{Has_Far_OPQSO_e} Akin H. and Mukhamedov F. Orthogonality preserving infinite dimensional quadratic stochastic
operators, \textit{AIP Conf. Proc.} {\bf 1676}(2015), 020008.

\bibitem{AR} Arambasic L., Rajic R. Operators preserving the strong
Birkhoff--James orthogonality on $B(H)$, \textit{Linear Algebra
Appl.} {\bf 471}(2015), 394--404.

\bibitem{Arut2012} Arutyunov A. V., Two problems of the theory of quadratic maps, \textit{Funct. Anal. Appl.} {\bf 46}(2012) 225--227

    \bibitem{1} Bernstein S. N. The solution of a mathematical problem concerning the theory of heredity,   \emph{Ann. Math. Statistics} {\bf 13}(1924) 53--61

\bibitem{Bur} Burgos M. Orthogonality preserving linear maps on C.-algebras with
non-zero socles, \textit{J. Math. Anal. Appl.} {\bf 401}(2013),
479--487.


\bibitem{F} Fosner A. Order preserving maps on the poset of upper triangular idempotent
matrices, \textit{Linear Algebra Appl.} {\bf 403}(2005), 248--262.

\bibitem{FS} Fosner A., Semrl P. Additive maps on matrix algebras
preserving invertibility or singularity, \textit{Acta Math. Sin.
(Engl. Ser.)} {\bf 21} (2005), 681–-684.

    \bibitem{GGZ} Ganikhodjaev N. N, Ganikhodjaev R. N. and Jamilov U. U. Quadratic stochastic operators and
zero-sum game dynamics, \textit{Ergod. Th. and Dynam.
        Sys.} {\bf 35}(2016), 1443--1473

    \bibitem{R.gani_tournment} Ganikhodzhaev  R.  N. Quadratic stochastic operators, Lyapunov functions and tournaments,  \textit{Sbornik Math.} {\bf 76}(1993)
    489--506.

    \bibitem{6} Ganikhodzhaev R., Mukhamedov F. and Rozikov U. Quadratic stochastic operators and processes: results and open problems, \emph{Infin.
Dimens. Anal. Quant. Prob. Relat. Top.} {\bf 14}(2011) 279--335



    \bibitem{8} Hofbauer J. and Sigmund K. \textit{Evolutionary Games and Population Dynamics}, Cambridge: Cambridge University
    Press, 1998

\bibitem{Horn} Horn  R. A. Li, C.-K., Merino D. I. Linear operators
preserving complex orthogonal equivalence on matrices, \textit{SIAM
J. Matrix Anal. Appl.} {\bf 15} (1994), 519-–529.

    \bibitem{J2013} Jamilov U. U. Quadratic stochastic operators corresponding to graphs, \textit{Lobachevskii  J. Math} {\bf 34}(2013), 148--151

    \bibitem{11} Lyubich Yu. I.  \textit{Mathematical structures in population
        genetics}, Berlin, Springer-Verlag, 1992.

    \bibitem{M2000} Mukhamedov  F. M.  On infinite dimensional Volterra operators,
\textit{Russian Math. Surveys} {\bf 55}(2000), 1161--1162

    \bibitem{Far_Has_Temir} Mukhamedov  F., Akin H. and Temir S.  On infinite dimensional
quadratic Volterra operators, \textit{J. Math. Anal. Appl.} {\bf
310}(2005), 533--556.

    \bibitem{MG2015}  Mukhamedov F. and Ganikhodjaev N. \textit{Quantum Quadratic Operators and
    Processes}, Berlin, Springer, 2015

    \bibitem{taha}  Mukhamedov F. and Taha M. H., On Volterra and
orthoganality preserving quadratic stochastic
operators,\textit{Miskloc Math. Notes} {\bf 17}(2016) 457--470

    \bibitem{20} Narendra S. G., Samaresh C. M. and Elliott W.M. On the Volterra and other nonlinear moldes of interacting
populations, \textit{Rev. Mod. Phys.} {\bf 43}(1971), 231-276

    \bibitem{21} Plank M, and Losert V. Hamiltonian structures for the n-dimensional Lotka-Volterra equations, \textit{J. Math. Phys.} {\bf 36}(1995), 3520-3543

\bibitem{man(2016)2d} Saburov M., On the surjectivity of quadratic stochastic operators acting on the simplex, \textit{Math. Notes} {\bf 99}(2016) 623-627

    \bibitem{25} Takeuchi Y. \textit{Global dynamicsal properties of Lotka-Volterra systems}, Singapore: World Scientific, 1996

    \bibitem{29} Volterra V. Lois de fluctuation de la population de plusieurs esp`eces coexistant dans le mˆeme
milieu, \textit{Association Franc¸aise Lyon} vol. 1926, pp. 96–98,
1927

\bibitem{W} Wolff M., Disjointness preserving operators in $C^*$-algebras, \textit{Arch.
Math.} {\bf 62} (1994,) 248--253.

\end{thebibliography}
\end{document}